\documentclass[english, 12pt]{amsproc}
\usepackage[english]{babel}
\usepackage{a4wide}
\usepackage{graphics}
\usepackage{amsfonts, amssymb, amscd, amsmath,amsthm}
\usepackage{latexsym}
\usepackage[matrix,arrow,curve]{xy}
\usepackage{mathabx}
\usepackage{color}
\usepackage{tikz}
\usetikzlibrary{matrix}

\DeclareMathOperator{\link}{link}
 
\DeclareMathOperator{\Ker}{Ker}

\DeclareMathOperator{\im}{Im}

\DeclareMathOperator{\Ann}{Ann} 
 
 \DeclareMathOperator{\GL}{GL}

 \DeclareMathOperator{\Hilb}{Hilb}
\DeclareMathOperator{\var}{var}

\newcommand{\Zo}{\mathbb{Z}}
\newcommand{\Ro}{\mathbb{R}}

\newcommand{\eqd}{\stackrel{\text{\tiny def}}{=}}

\newcommand{\I}{\mathcal{I}}
\newcommand{\D}{\mathcal{D}}
\newcommand{\A}{\mathcal{A}}

\newcommand{\dd}{\partial}

\newcounter{stmcounter}[section]

\numberwithin{equation}{section}

\theoremstyle{plain}
\newtheorem{cor}[stmcounter]{Corollary}

\newtheorem{thm}[stmcounter]{Theorem}

\newtheorem{prop}[stmcounter]{Proposition}
\newtheorem{lemma}[stmcounter]{Lemma}
\newtheorem{problem}[stmcounter]{Problem}

\newtheorem{claim}[stmcounter]{Claim}

\theoremstyle{definition}
\newtheorem{defin}[stmcounter]{Definition}

\theoremstyle{remark}
\newtheorem{ex}[stmcounter]{Example}
\newtheorem{rem}[stmcounter]{Remark}
\newtheorem{con}[stmcounter]{Construction}

\begin{document}

\title{Dimensions of multi-fan algebras}

\author{Anton Ayzenberg}
\thanks{This work is supported by the Russian Science Foundation under grant
 14-11-00414}
\address{Steklov Mathematical Institute of Russian Academy of Sciences, Moscow}
\email{ayzenberga@gmail.com}



\begin{abstract}
Given an arbitrary non-zero simplicial cycle and a generic vector
coloring of its vertices, there is a way to produce a graded
Poincare duality algebra associated to these data. The procedure
relies on the theory of volume polynomials and multi-fans. This
construction includes many important examples, such as cohomology
of toric varieties and quasitoric manifolds, and Gorenstein
algebras of triangulated homology manifolds, introduced by Novik
and Swartz. In all these examples the dimensions of graded
components of such duality algebras do not depend on the vector
coloring. It was conjectured that the same holds for any
simplicial cycle. We disprove this conjecture by showing that the
colors of singular points of the cycle may affect the dimensions.
However, the colors of smooth points are irrelevant. By using
bistellar moves we show that the number of different dimension
vectors arising on a given 3-dimensional pseudomanifold with
isolated singularities is a topological invariant. This invariant
is trivial on manifolds, but nontrivial in general.
\end{abstract}

\maketitle

\section{Introduction}

A multi-fan \cite{Mas} is a collection of full-dimensional convex
cones in the oriented space $V\cong\Ro^n$, emanating from the
origin and equipped with weights. In contrast to the usual fans,
cones in a multi-fan may overlap. The condition of completeness
for a multi-fan is the following: the multi-fan $\Delta$ is
complete if the weighted sum of its maximal cones is a cycle. This
means that for every codimension one cone $C$ the weights of all
cones having $C$ as a face sum to zero, if we count them with
different signs depending on which side of $C$ they are located. A
multi-polytope based on a multi-fan $\Delta$ is a collection of
affine hyperplanes each of which is orthogonal to some ray of
$\Delta$ and intersecting whenever the corresponding rays are
faces of some cone of $\Delta$. In this work as well as in
\cite{AyM} we restrict to the situation when all cones of a
multi-fan $\Delta$ are simplicial: in this case $\Delta$ is called
simplicial multi-fan and corresponding multi-polytopes are called
simple. In the simplicial case we use the following definitions

\begin{defin}[\cite{Mas,AyM}]\label{defMFanMPoly}
\emph{A complete simplicial multi-fan} is a pair
$(\omega,\lambda)$, where
\[
\omega=\sum_{I\subset [m],
|I|=n}w(I)I\in Z_{n-1}(\triangle_{[m]}^{(n-1)})
\]
is a simplicial cycle on $m$ vertices, and $\lambda\colon [m]\to
V$ is any function satisfying the condition: $\{\lambda(i)\mid
i\in I\}$ is a basis of $V$ if $|I|=n$ and $w(I)\neq 0$. In this
case $\lambda$ is called \emph{a characteristic function}. If
$\omega$ is the fundamental cycle of $(n-1)$-dimensional oriented
pseudomanifold $K$, we say that $\Delta$ is supported on $K$.

\emph{A simple multi-polytope} $P$ is a pair
$(\Delta,\{H_1,\ldots,H_m\})$, where $\Delta$ is a simplicial
multi-fan, and $H_i$ is a hyperplane in $V^*$ orthogonal to
$\lambda(i)\in V$, that is
\[
H_i=\{x\in V^*\mid \langle x,\lambda(i)\rangle=c_i\}.
\]
We say that $P$ is based on $\Delta$. The numbers
$c_1,\ldots,c_m\in \Ro$ are called \emph{the support parameters}
of a multi-polytope $P$.
\end{defin}

Every simple multi-polytope based on a given simplicial multi-fan
is uniquely characterized by the set of its support parameters.

Many common notions and facts about convex polytopes and fans are
naturally extended to multi-fans and multi-polytopes. In
particular, whenever $P$ is a multi-polytope based on a complete
multi-fan, there is a well-defined notion of volume of $P$, see
\cite{HM}. For simple multi-polytopes there is a formula for
volume, generalizing Lawrence formula for simple polytopes
\cite{AyM,Law}. Considering volumes of all multi-polytopes based
on a fixed multi-fan $\Delta$ at once, we get the volume function
$V_\Delta\colon \Ro^m\to\Ro$ defined on the vector space of
support parameters. To a tuple $(c_1,\ldots,c_m)\in\Ro^m$ this
function associates the volume of the multi-polytope with support
parameters $(c_1,\ldots,c_m)$ based on $\Delta$. The function
$V_\Delta$ is a homogeneous polynomial of degree $n$ in the
support parameters: $V_\Delta\in \Ro[c_1,\ldots,c_m]_n$. It is
called \emph{the volume polynomial} of a multi-fan $\Delta$.

There is a standard procedure to make a Poincare duality algebra
out of any homogeneous polynomial, called Macaulay duality
\cite{MS}. In the case of volume polynomials it was introduced and
studied in \cite{PKh2,Tim}. Consider the graded ring of
polynomials $\D=\Ro[\dd_1,\ldots,\dd_m]$ where each symbol $\dd_i$
denotes $\frac{\dd}{\dd c_i}$, the partial derivative in $i$-th
support parameter. For topological reasons we double the degree
assuming $\deg\dd_i=2$. Each variable $\dd_i$ acts on
$\Ro[c_1,\ldots,c_m]$ in a natural way, so we may consider the
homogeneous ideal $\Ann V_\Delta=\{D\in\D\mid DV_\Delta=0\}$ of
differential operators annihilating $V_\Delta$. The quotient
algebra $\A^*(\Delta)\eqd \D/\Ann V_\Delta$ satisfies Poincare
duality conditions: its top component $\A^{2n}(\Delta)$ is
one-dimensional and the pairing
$\A^{2j}(\Delta)\otimes\A^{2n-2j}(\Delta)\stackrel{\times}{\to}
\A^{2n}(\Delta)$ is non-degenerate. We call $\A^*(\Delta)$ the
duality algebra of a multi-fan. Let $d_j$ denote the dimension of
the graded component $\A^{2j}$. Poincare duality implies
$d_j=d_{n-j}$.

Multi-fan algebras and the way they were constructed seem to be
very important. When $\Delta$ is just the ordinary complete
simplicial rational fan, the algebra $\A^*(\Delta)$ coincides with
the cohomology of the corresponding toric variety (see
\cite{PKh2}). Timorin \cite{Tim} gave a purely geometrical proof
of Stanley's g-theorem for a polytopal fan $\Delta$ by showing
that $\A^*(\Delta)$ satisfies Lefschetz property. In \cite{AyM} we
proved that whenever the multi-fan $\Delta$ is supported on an
orientable homology manifold $K$, the algebra $\A^*(\Delta)$
coincides with the Gorenstein algebra introduced by Novik and
Swartz in \cite{NSgor}. However, we also showed that every
(finite-dimensional commutative) Poincare duality algebra
generated in degree $2$ is isomorphic to $\A^*(\Delta)$ for some
multi-fan $\Delta$. Therefore there exist complete simplicial
multi-fans whose algebras do not satisfy Lefschetz property and
whose dimensions' sequence $(d_0,d_1,\ldots,d_n)$ is not unimodal.

By Definition \ref{defMFanMPoly} a complete simplicial multi-fan
$\Delta$ have two pieces of information: a simplicial cycle
$\omega$, the combinatorial data, and a characteristic function
$\lambda\colon [m]\to V$, which encodes the directions of rays of
a multi-fan, the geometrical data. In many cases the dimensions of
graded components $d_j=\dim\A^{2j}(\Delta)$ depend only on
$\omega$, but not on $\lambda$. When $\Delta$ is supported on a
sphere, the dimensions coincide with the h-numbers of this sphere.
More generally, if $\Delta$ is supported on a manifold, the
dimensions $d_j$ coincide with the so called h''-numbers of a
manifold, which are the combinatorial invariants of this manifold
(see \cite{NSgor}). We suggested this was a general phenomenon
\cite[Conj.1]{AyM}.

\begin{problem}
Is it true that dimensions $(d_0,d_1,\ldots,d_n)$ of a multi-fan
algebra $\A^*(\Delta)$ depend only on $\omega$ but not on
characteristic function $\lambda$? If no, is it true for
multi-fans supported on pseudomanifolds?
\end{problem}

In this paper we answer both questions in negative by providing
two counter-examples. First counter-example is computed by hand
and relies on several simple facts proved previously in
\cite{AyM}. However, this counter-example is not a pseudomanifold
so it does not answer the second question. Second question is more
complicated since, in a sense, the easiest interesting example of
a pseudomanifold which is not a manifold is the suspension over a
2-torus. In this case the volume polynomial is a polynomial of
degree 4, and calculations can hardly be made by hand. However a
simple programm written in GAP \cite{GAP4,simpcomp} allowed to
answer the second question. If $K$ is the suspension over the
minimal triangulation of a torus, then there exist two multi-fans
supported on $K$ having d-vectors $(1,5,8,5,1)$ and
$(1,5,12,5,1)$. Suspension over all other triangulations of a
2-torus, as well as suspensions over other orientable surfaces,
are also discussed in the paper.

In a more theoretical part of this work we explain what makes
manifolds so special from the view point of multi-fans. Let
$\omega$ be a simplicial cycle and $i\in [m]$ be its vertex. The
link of the vertex $i$ in a cycle $\omega$ may be defined in a
natural way. We say that $i$ is a smooth vertex of a cycle
$\omega$ if the link of $i$ is (the fundamental cycle of) a
homology sphere.

\begin{thm}\label{thmSmoothPoints}
Dimensions $d_j$ of multi-fan algebra does not depend on the
values of characteristic function in smooth vertices.
\end{thm}

In particular, since all vertices of a triangulated manifold are
smooth, d-vectors of multi-fans supported on manifolds do not
depend on characteristic function at all. Nevertheless, our
examples show that d-vectors of mutli-fans may crucially depend on
the values of $\lambda$ in singular points.

It is natural to consider the following invariant of a simplicial
pseudomanifold $K$: the number $r(K)$ of distinct dimension
vectors of multi-fans supported on $K$. By the preceding
discussion, this invariant equals $1$ on homology manifolds, but
it is nontrivial in general. Using bistellar moves we show that
this number is a topological invariant on the class of
3-dimensional pseudomanifolds with isolated singularities. The
properties of this invariant are yet unknown.

The following family of examples is considered: take an arbitrary
link $l\colon\bigsqcup_\alpha S_\alpha^1\hookrightarrow S^3$ and
collapse each of its components to a point. We conjecture that the
resulting pseudo-manifold $K$ has the property $r(K)=1$ whenever
all components of the link are pairwise unlinked. We checked that
$r(K)=1$ for disjoint union of knots and for Borromean rings. For
Hopf link we have $r(K)=2$ (this is just the case of suspension
over a 2-torus).

Our considerations suggest that the additive structure of
multi-fan algebras over pseudomanifolds may lead to new invariants
of 3-pseudomanifolds or, at least, uncover interesting connections
between convex geometry and 3-dimensional topology.

\section{Necessary notions and facts}\label{secPrelim}

Let $\Psi\in \Ro[c_1,\ldots,c_m]_n$ be an arbitrary non-zero
homogeneous polynomial of degree $n$ and let $\D^*/\Ann\Psi$ be
the corresponding Poincare duality algebra (i.e.
$\D^*=\Ro[\dd_1,\ldots,\dd_m]$, $\dd_i=\frac{\dd}{\dd c_i}$,
$\Ann\Psi=\{D\in\D^*\mid D\Psi=0\}$). Let $\var^j(\Psi)\subset
\Ro[c_1,\ldots,c_m]_{n-j}$ be the linear span of all partial
derivatives of degree $j$ of the polynomial $\Psi$. The linear map
$(\D/\Ann\Psi)_{2j}\to \var^j\Psi$, $D\mapsto D\Psi$ is an
isomorphism.


Let $\Delta=(\omega,\lambda)$ be a complete simplicial multi-fan
on a finite set $M=[m]$ of rays in the space $V\cong\Ro^n$, where
$\omega=\sum_{I\subset M,|I|=n}w(I)I$, $\lambda\colon M\to V$.
Consider the simplicial complex $K$ on $M$ whose maximal simplices
are all subsets $I\subset M,|I|=n$ such that $w(I)\neq 0$. $K$ is
called \emph{the support} of the cycle $\omega$; it has dimension
$n-1$. If $J\in K$ is a simplex of any dimension, then we can
define the projected multi-fan $\Delta_J$ as follows.

Consider the link of $J$ in $K$: $\link_KJ=\{I\subset[m]\mid I\cap
J=\varnothing, I\sqcup J\in K\}$, and let $M_J\subset M$ be the
set of vertices of $\link_KJ$. Let $V_J$ be the quotient of the
vector space $V$ by the subspace $\langle\lambda(i)\mid i\in
J\rangle\cong\Ro^{|J|}$. Consider the simplicial cycle
\[
\omega_J=\sum_{I\in \link_KJ,|I|=n-|J|}w(I\sqcup J)I\in
Z_{n-1-|J|}(\link_KJ;\Ro).
\]
The projected characteristic function $\lambda_J\colon M_J\to V$
is defined as the composition $M_J\subset M\stackrel{\lambda}{\to}
V\twoheadrightarrow V_J$, where the last arrow is the natural
projection. The multi-fan $\Delta_J=(\omega_J,\lambda_J)$ is
called \emph{the projected multi-fan} of $\Delta$ with respect to
$J$.

Let $V_\Delta\in \Ro[c_1,\ldots,c_m]_n$ be the volume polynomial
of $\Delta$. For a subset $J=\{i_1,\ldots,i_j\}\subset[m]$ let
$\dd_J$ denote the differential operator
$\dd_{i_1}\cdots\dd_{i_j}$. In \cite[Lm.1]{AyM} we proved that
$\dd_JV_\Delta$ is zero whenever $J\notin K$; otherwise
$\dd_JV_\Delta$ coincides with $V_{\Delta_J}$, the volume
polynomial of the projected multi-fan, up to epimorphic linear
change of variables and up to constant factor. In particular, we
have

\begin{equation}\label{eqVariationsEq}
\dim \var^s\dd_JV_\Delta=\dim \var^sV_{\Delta_J} \mbox{ for any }
s=0,\ldots,n.
\end{equation}

Let us recall the formula for the volume polynomial.

\begin{prop}[\cite{AyM}]
Let $\Delta=(\omega,\lambda)$ be as above and $v\in V$ be a
generic vector (it will be called the polarization vector). Then
\begin{equation}\label{eqVolPolyFormula}
V_{\Delta}(c_1,\ldots,c_m)=\frac{1}{n!}\sum_{I=\{i_1,\ldots,i_n\}\in
K} \frac{w(I)}{|\det\lambda(I)|\prod_{j=1}^n\alpha_{I,j}}
(\alpha_{I,1}c_{i_1}+\cdots+\alpha_{I,n}c_{i_n})^n,
\end{equation}
where $\alpha_{I,1},\ldots,\alpha_{I,n}$ are the coordinates of
$v$ in the basis $(\lambda(i_1),\ldots,\lambda(i_n))$, $w(I)$ is
the weight of the simplex $I$, and $\det\lambda(I)$ is the
determinant of the matrix $(\lambda(i_1),\ldots,\lambda(i_n))$.
\end{prop}

The condition that $v$ is generic means that all coefficients
$\alpha_{I,i}$ are non-zero: this is an open condition in $V$.

\begin{rem}\label{remLinearChange}
It can be seen from this formula that whenever the operator
$A\in\GL(V)$ acts on all values of characteristic function
simultaneously, the volume polynomial does not change up to
constant factor. Indeed, if we take $Av$ as a polarization vector
for the multi-fan $A\Delta=(\omega, A\lambda)$, all the
coefficients $\alpha_{I,i}$ remain unchanged, and all the
determinants $\det\lambda_I$ are multiplied by the same factor
$\det A$. Therefore, $\A^*(A\Delta)\cong \A^*(\Delta)$.
\end{rem}

We finish this section with a small remark, which will be used in
the following.

\begin{rem}\label{remGhostAndSums}
Let us fix a finite set $[m]$ and a function $\lambda\colon [m]\to
V\cong\Ro^n$ which is general enough. All $n$-dimensional
multi-fans on $[m]$ having $\lambda$ as a characteristic function
form a vector space (essentially, this is just a certain subspace
of the space of all $(n-1)$-cycles on $[m]$ vertices). Therefore,
one can form sums and differences of multi-fans, provided that
they have the same vertex sets and characteristic functions.
Volume polynomial is additive with respect to this operation,
which easily follows from its formula.

It may happen that the underlying cycle of a multi-fan does not
contain some vertices from $M$. We call such vertices ghost
vertices. The polynomial $V_\Delta$ does not actually depend on
the variable $c_i$ for any ghost vertex $i\in M$. In this case
$\dd_i=0$ in $\A^*(\Delta)$.
\end{rem}

\section{Values of characteristic function in smooth
points}\label{secSmoothPoints}

In this section we prove Theorem \ref{thmSmoothPoints}.

\begin{defin}
A simplicial cycle $\omega\in Z_{n-1}(\triangle_{[m]}^{n-1};\Ro)$
is called \emph{rigid} if the dimensions $d_j=\dim\A^{2j}(\Delta)$
of all multi-fans $\Delta=(\omega,\lambda)$ are independent of
$\lambda$.
\end{defin}

As was mentioned in the introduction, the fundamental cycle of any
oriented homology manifold $K$ is rigid since $d_j=h_j''(K)$. In
particular, every homology sphere $K$ is rigid and $d_j=h_j(K)$.
If $\omega$ is rigid, we denote the dimension $\dim\A^{2j}$ by
$d_j(\omega)$.

\begin{con}
Let $\omega'=\sum_{|I'|=n'}w'(I')I'$,
$\omega''=\sum_{|I''|=n''}w''(I'')I''$ be two simplicial cycles on
disjoint vertex sets $M'$, $M''$. Define the join $\omega'\ast
\omega''$ as a simplicial cycle on $M'\sqcup M''$:
\[
\omega'\ast \omega'' = \sum_{\substack{I'\in M', |I'|=n'\\I''\in
M'', |I''|=n''}}\omega'(I')\omega''(I'')I'\sqcup I''\in
Z_{n'+n''-1}(\triangle_{M'\sqcup M''}^{n'+n''-1};\Ro).
\]

Let us define the join of two multi-fans. Let
$\Delta'=(\omega',\lambda')$ and $\Delta''=(\omega'',\lambda'')$
be multi-fans in the spaces $V'$ and $V''$ with the ray-sets $M'$
and $M''$ respectively. Consider the multi-fan
$\Delta'\ast\Delta''=(\omega'\ast\omega'',\lambda)$, where
$\lambda\colon M'\sqcup M''\to V'\oplus V''$ is given by
\[
\lambda(i)=\begin{cases} (\lambda'(i),0), \mbox{ if }i\in M',\\
(0,\lambda''(i)), \mbox{ if } i\in M''.
\end{cases}
\]
There holds
\[
V_{\Delta'\ast\Delta''}=V_{\Delta'}\cdot V_{\Delta''}.
\]
This can be deduced either from the exact formula of the volume
polynomial or geometrically, by noticing that every multi-polytope
based on $\Delta$ is just the cartesian product of a
multi-polytope based on $\Delta'$ and a multi-polytope based on
$\Delta''$, so the volumes are multiplied. The polynomials
$V_{\Delta'}$ and $V_{\Delta''}$ have distinct sets of variables,
which implies
\begin{equation}\label{eqTensorProd}
\A^*(\Delta'\ast\Delta'')\cong \A^*(\Delta')\otimes
\A^*(\Delta'').
\end{equation}
Therefore,
\begin{equation}\label{eqJoinMFansHilb}
\Hilb(\A^*(\Delta'\ast\Delta'');t)=\Hilb(\A^*(\Delta');t)\cdot
\Hilb(\A^*(\Delta'');t).
\end{equation}
\end{con}

Let $S^0$ denote the simplicial complex consisting of two disjoint
vertices $x$ and $y$. By abuse of notation we use the same symbol
$S^0$ to denote its underlying simplicial cycle lying in
$Z_0(S^0;\Ro)$. The join of a cycle $\omega$ with $S^0$ is called
the suspension and is denoted by $\Sigma\omega$.

$\Delta$ is called \emph{a suspension-shaped multi-fan} if its
underlying simplicial cycle is isomorphic to $\Sigma\omega$ for
some cycle $\omega$. Note that the algebra of a suspension-shaped
multi-fan contains two marked elements $\dd_x,\dd_y\in
\A^*(\Delta)$ corresponding to the apices of the suspension. Since
$\{x,y\}$ does not lie in the support of $\Sigma\omega$ we have
the relation $\dd_x\dd_y=0$ in $\A^*(\Delta)$. Let us consider two
operators
\[
\times\dd_x,\times\dd_y\colon \A^*(\Delta)\to \A^{*+2}(\Delta),
\]
acting on the multi-fan algebra. There holds
$\im(\times\dd_y)\subseteq\Ker(\times\dd_x)$.

\begin{defin}
A suspension-shaped multi-fan $\Delta$ is called \emph{editable},
if
\[
\im(\times\dd_y)=\Ker(\times\dd_x).
\]
\end{defin}

Our next goal is to prove that suspensions over homology spheres
are editable. Several technical lemmas are needed.

\begin{lemma}\label{lemOrthogonalIdeals}
Let $\A^*$ be a Poincare duality algebra of formal dimension $2n$
and $\I\subset \A^*$ be a graded ideal. Let $\Ann\I:=\{a\in
\A^*\mid a\I=0\}$ and let $\I^\bot=\bigoplus_j(\I^\bot)^{2j}$
denote the component-wise orthogonal complement:
\[
(\I^\bot)^{2j}=\{a\in \A^{2j}\mid a\I^{2n-2j}=0\}.
\]
Then $\Ann\I=\I^\bot$.
\end{lemma}

The proof is straightforward. We call ideals $\I$ and
$\Ann\I=\I^\bot$ orthogonal. If $\Delta$ is a suspension-shaped
multi-fan, then the ideal $\im(\times\dd_x)$ (which is just the
principal ideal of $\A^*(\Delta)$ generated by $\dd_x$) is
orthogonal to $\Ker(\times\dd_x)$ by definition.

\begin{lemma}\label{lemHilbLoc}
$\Hilb(\im(\times\dd_x);t) = t^2\Hilb(\A^*(\Delta_x);t)$.
\end{lemma}

\begin{proof}
Recall from \S\ref{secPrelim} that there is an isomorphism
$\A^{2j}(\Delta)\to\var^jV_\Delta$ which sends $D$ to
$DV_{\Delta}$. Therefore
\[
\dim \im(\times\dd_x)_{2j}=\dim\{DV_{\Delta}\mid D\in
\im(\times\dd_x)_{2j}\}.
\]
The latter space may be identified with
\[
\{D\dd_xV_{\Delta}\mid D\in \D_{2j-2}\}=\var^{j-1}(\dd_xV_\Delta).
\]
Equation \eqref{eqVariationsEq} implies that
$\dim\var^{j-1}(\dd_xV_\Delta) = \dim\var^{j-1}V_{\Delta_x} =
\dim\A^{2j-2}(V_{\Delta_x})$. This finishes the proof.
\end{proof}

\begin{lemma}\label{lemDoublyRigidIsValid}
Suppose that a simplicial $(n-2)$-cycle $\omega$ is rigid and its
suspension $\Sigma\omega$ is rigid. Then any suspension-shaped
multi-fan $\Delta$ on $\Sigma\omega$ is editable.
\end{lemma}

\begin{proof}
For any multi-fan $\Delta'$ based on $S^0$ we have
$\Hilb(\Delta';t)=1+t^2$ (since $S^0$ is a sphere and its
h-numbers are $(1,1)$). Since $\omega$, $S^0$, and $\Sigma\omega$
are rigid, formula \eqref{eqJoinMFansHilb} implies
\[
\Hilb(\A^*(\Delta);t)=(1+t^2)\sum_{j=0}^{n-1}d_j(\omega)t^{2j}.
\]
By Lemma \ref{lemHilbLoc},
\[
\Hilb(\im(\times\dd_x);t)=t^2\Hilb(\A^*(\Delta_x);t).
\]
Since $\Delta_x$ is a multi-fan based on $\omega$, there holds
$\Hilb(\im(\times\dd_x);t) =
t^2\sum_{j=0}^{n-1}d_j(\omega)t^{2j}$. Similarly,
$\Hilb(\im(\times\dd_y);t) =
t^2\sum_{j=0}^{n-1}d_j(\omega)t^{2j}$. Using Lemma
\ref{lemOrthogonalIdeals}, we may find the dimensions of the
orthogonal complement $\Ker(\times\dd_x)=\im(\times\dd_x)^\bot$ in
each degree:
\[
\dim\Ker(\times\dd_x)_{2j}=\dim\A^{2j}(\Delta)-\dim\im(\times\dd_x)_{2n-2j}.
\]
This implies
\[
\Hilb(\Ker(\times\dd_x);t)=t^2\Hilb(\A^*(\Delta_x);t)=\Hilb(\im(\times\dd_y);t).
\]
Thus $\Ker(\times\dd_x)=\im(\times\dd_y)$.
\end{proof}

\begin{cor}\label{corSphereSuspIsValid}
Let $K$ be a homology $(n-2)$-sphere (or its underlying simplicial
cycle). Then any suspension-shaped multi-fan $\Delta$ on $\Sigma
K$ is editable.
\end{cor}

\begin{proof}
Suspension over a homology sphere is again a homology sphere. Thus
both $K$ and $\Sigma K$ are rigid and Lemma
\ref{lemDoublyRigidIsValid} applies.
\end{proof}

\begin{rem}
The argument used in the proof of Lemma
\ref{lemDoublyRigidIsValid} shows that in general, if $\Delta$ is
a suspension-shaped multi-fan with suspension points $x, y$, there
holds
\[
\Hilb(\A^*(\Delta);t)\geqslant
\Hilb(\A^*(\Delta_x);t)+t^2\Hilb(\A^*(\Delta_y);t)
\]
\end{rem}

The next construction shows that suspension-shaped multi-fans
arise naturally when we change the value of characteristic
function at a single point.

\begin{con}
Let $\Delta'=(\omega,\lambda')$,$\Delta''=(\omega,\lambda'')$ be
multi-fans based on the same simplicial cycle $\omega$, and assume
that $\lambda'(i)=\lambda''(i)$ for all $i\in[m]$ except $i=k$,
where $k$ is a fixed vertex. It is convenient to take two copies
$x,y$ of $k$ and consider $\Delta'$ and $\Delta''$ as the
multi-fans on the set $M:=([m]\setminus \{k\})\sqcup\{x,y\}$
having the same characteristic function $\lambda$:
\[
\lambda(i)=\begin{cases}\lambda'(i)=\lambda''(i), \mbox{ if }
i\in[m]\setminus \{k\}, \\\lambda'(k), \mbox{ if } i=x,\\
\lambda''(k), \mbox{ if } i=y.
\end{cases}
\]
The underlying cycles of $\Delta'$ and $\Delta''$ are isomorphic,
but they are different as the elements of
$Z_{n-1}(\triangle_{M}^{n-1};\Ro)$. The cycle $\omega'$ passes
through $x$ and has ghost vertex $y$, while the cycle $\omega''$
passes through $y$ and has ghost vertex $x$. Since $\Delta'$,
$\Delta''$ have the same characteristic function, their difference
is well-defined:
\[
T=\Delta''-\Delta'.
\]
It is easy to observe that $T$ is a suspension-shaped multi-fan
with suspension points $x$ and $y$, whose underlying simplicial
cycle has the form $\Sigma\omega_x=\Sigma\omega_y$ (recall that
$\omega_x$, $\omega_y$ are the projected simplicial cycles with
respect to vertices $x$ and $y$ respectively).
\end{con}

\begin{thm}\label{thmLocalRigidity}
Let $\Delta'$, $\Delta''$ be as above. If the projected
$(n-2)$-cycle $\omega_k$ is rigid and $T=\Delta''-\Delta'$ is an
editable suspension-shaped multi-fan, then
$\Hilb(\A^*(\Delta');t)=\Hilb(\A^*(\Delta'');t)$.
\end{thm}

\begin{proof}
Consider the ring of differential operators $\D^*=\Ro[\dd_i\mid
i\in M]$ and the principal ideals $(\dd_x)$, $(\dd_y)$ in this
ring. Note that $(\dd_y)\subset\Ann V_{\Delta'}$ since $y$ is a
ghost vertex of $\Delta'$, and similarly $(\dd_x)\subset\Ann
V_{\Delta''}$.

\begin{claim}
$\Ann V_{\Delta'}+(\dd_x)=\Ann V_{\Delta''}+(\dd_y)$
\end{claim}

It is enough to show that $\Ann V_{\Delta'}\subset
V_{\Delta''}+(\dd_y)$ (the symmetry of the statement would imply
$\Ann V_{\Delta''}\subset V_{\Delta'}+(\dd_x)$).

First note that $V_{\Delta''}=V_{\Delta'}+V_T$. Hence
$\dd_xV_{\Delta'}+\dd_xV_T=\dd_xV_{\Delta''}=0$. Consider $D\in
\Ann V_{\Delta'}$. Then $\dd_xDV_T=-D\dd_xV_{\Delta'}=0$.
Therefore the class of $D$ in the algebra $\A^*(T)$ lies in the
kernel of $\times\dd_x\colon\A^*(T)\to\A^{*+2}(T)$. By assumption,
$T$ is editable, which implies $D\in \Ann V_T+(\dd_y)$. Thus
$D=D_1+D_2$ where $D_1\in \Ann V_T$ and $D_2\in (\dd_y)$.

Applying $D_1=D-D_2$ to the polynomial
$V_{\Delta''}=V_T+V_{\Delta'}$ we get
\[
D_1V_{\Delta''}=D_1(V_T)+(D-D_2)V_{\Delta'}=0
\]
since $D\in\Ann V_{\Delta'}$ by assumption and $D_2\in
(\dd_y)\subset\Ann V_{\Delta'}$. Thus $D=D_1+D_2$ where $D_1\in
\Ann V_{\Delta''}$ and $D_2\in (\dd_y)$ which proves the claim.

Now consider the diagram of inclusions of ideals:

\begin{equation}\label{eqDiagrOfIdeals}
\xymatrix{\Ann V_{\Delta'}\ar@{^{(}->}[r]&\Ann
V_{\Delta'}+(\dd_x)=\Ann V_{\Delta''}+(\dd_y)
&\Ann V_{\Delta''}\ar@{_{(}->}[l]\\
\Ann V_{\Delta'}\cap(\dd_x)\ar@{^{(}->}[r]\ar@{^{(}->}[u]&
(\dd_x,\dd_y)\ar@{^{(}->}[u]& \Ann
V_{\Delta''}\cap(\dd_y)\ar@{^{(}->}[u]\ar@{_{(}->}[l] }
\end{equation}

\begin{claim}
$\Hilb(\Ann V_{\Delta'}\cap(\dd_x);t)=\Hilb(\Ann
V_{\Delta''}\cap(\dd_y);t)$.
\end{claim}

We have $\Ann V_{\Delta'}\cap(\dd_x)=\{\dd_xD\in \D\mid
\dd_xDV_{\Delta'}=0\}$. Up to the shift of grading this vector
space coincides with $\Ann(\dd_xV_{\Delta'})$. There holds
$\Hilb(\Ann(\dd_xV_{\Delta'});t)=\Hilb(\D;t)-\Hilb\A^*(\dd_xV_{\Delta'})$.
Note that
$\dim\A^{2j}(\dd_xV_{\Delta'})=\dim\var^j(\dd_xV_{\Delta'})$ and
the latter space has the same dimension as
$\dim\var^j(V_{\Delta'_x})=\dim \A^{2j}(\Delta'_x)$ according to
equation \eqref{eqVariationsEq}. Both projected multi-fans
$\Delta'$ and $\Delta''$ are based on the same projected cycle
$\omega_k$. By the assumption of the theorem, $\omega_k$ is rigid,
thus $\dim \A^{2j}(\Delta'_x)=\dim \A^{2j}(\Delta''_y)$. This
proves the claim.

Finally, the diagram \eqref{eqDiagrOfIdeals} and the claim imply
$\Hilb(\Ann V_{\Delta'};t)=\Hilb(\Ann V_{\Delta''};t)$ which
proves the theorem.
\end{proof}

Corollary \ref{corSphereSuspIsValid} and Theorem
\ref{thmLocalRigidity} together imply Theorem
\ref{thmSmoothPoints}.

\section{Singular examples}\label{secCounterExs}

\begin{prop}\label{propCounterEx1}
There exist two complete simplicial multi-fans $\Delta_1,
\Delta_2$ having the same underlying simplicial cycle but
different dimension vectors.
\end{prop}

\begin{proof}
Let $m=6$ and $n=2$. Consider the oriented graph $\Gamma$ depicted
in Fig.\ref{figGamma}. Let $\omega_{\Gamma}\in
C_1(\triangle_{[6]};\Zo)$ be the simplicial chain which is the sum
of all oriented edges of $\Gamma$ with weights $1$. Since $\Gamma$
is eulerian, $\omega_{\Gamma}$ is a cycle.

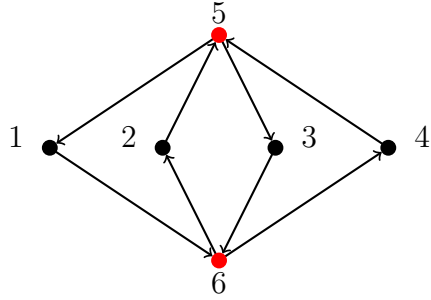
\begin{figure}[h]
\begin{center}
    \begin{tikzpicture}[scale=1.5]
        \draw[->,shorten >=3pt, thick] (0,0)--(1.5,-1);
        \draw[->,shorten >=3pt, thick] (1.5,-1)--(1,0);
        \draw[->,shorten >=3pt, thick] (1,0)--(1.5,1);
        \draw[->,shorten >=3pt, thick] (1.5,1)--(0,0);
        \draw[->,shorten >=3pt, thick] (2,0)--(1.5,-1);
        \draw[->,shorten >=3pt, thick] (1.5,-1)--(3,0);
        \draw[->,shorten >=3pt, thick] (3,0)--(1.5,1);
        \draw[->,shorten >=3pt, thick] (1.5,1)--(2,0);
        \fill (0,0) circle (2pt);
        \draw (-0.3,0.1) node{$1$};
        \fill (1,0) circle (2pt);
        \draw (0.7,0.1) node{$2$};
        \fill (2,0) circle (2pt);
        \draw (2.3,0.1) node{$3$};
        \fill (3,0) circle (2pt);
        \draw (3.3,0.1) node{$4$};
        \fill[red] (1.5,1) circle (2pt);
        \draw (1.5,1.2) node{$5$};
        \fill[red] (1.5,-1) circle (2pt);
        \draw (1.5,-1.2) node{$6$};
    \end{tikzpicture}
\end{center}
\caption{Graph $\Gamma$} \label{figGamma}
\end{figure}

Let us define two complete simplicial multi-fans
$\Delta_1,\Delta_2$ with underlying cycle $\omega_\Gamma$. To do
this, we need to specify the values of characteristic functions
$\lambda_{1,2}\colon[6]\to\Ro^2$.

(1) Define the function $\lambda_1$ so that its values in the
vertices $1,2,3,4$ lie on $x$-axis and the values in vertices
$5,6$ lie in $y$-axis. For example, take
$\lambda_1(1)=\lambda_1(2)=\lambda_1(3)=\lambda_1(4)=e_1$,
$\lambda_1(5)=\lambda_1(6)=e_2$, where $e_1,e_2$ is the basis of
$\Ro^2$. In this case the multi-fan
$\Delta_1=(w_\Gamma,\lambda_1)$ is the join of two 1-dimensional
multi-fans $\Delta'$ and $\Delta''$ depicted below:
\begin{center}
    \begin{tikzpicture}[scale=1.5]
        \fill (0,0) circle (2pt);
        \draw (0,0.2) node{$e_1$};
        \draw (0,-0.2) node{$-1$};
        \fill (1,0) circle (2pt);
        \draw (1,0.2) node{$e_1$};
        \draw (1,-0.2) node{$+1$};
        \fill (2,0) circle (2pt);
        \draw (2,0.2) node{$e_1$};
        \draw (2,-0.2) node{$-1$};
        \fill (3,0) circle (2pt);
        \draw (3,0.2) node{$e_1$};
        \draw (3,-0.2) node{$+1$};

        \draw (3.5,0) node{$*$};

        \fill[red] (4.5,0.5) circle (2pt);
        \draw (4.3,0.5) node{$e_2$};
        \draw (4.8,0.5) node{$+1$};
        \fill[red] (4.5,-0.5) circle (2pt);
        \draw (4.3,-0.5) node{$e_2$};
        \draw (4.8,-0.5) node{$-1$};
    \end{tikzpicture}
\end{center}
Equation \eqref{eqTensorProd} tells that
$\A^*(\Delta_1)=\A^*(\Delta')\otimes \A^*(\Delta'')$. Algebra of
any 1-dimensional multi-fan has Hilbert function $1+t^2$ due to
Poincare duality. Finally,
$\Hilb(\A^*(\Delta_1);t)=(1+t^2)(1+t^2)=1+2t^2+t^4$.

(2) Let us define the function $\lambda_2\colon [6]\to \Ro^2$. Set
the values $\lambda_2(1),\lambda_2(2),\lambda_2(3),\lambda_2(4)$
arbitrarily (for example we may set them equal to $e_1$). Now
choose $\lambda_2(5)$ and $\lambda_2(6)$ so that they are linearly
independent. Then $\Delta_2=(w_\Gamma,\lambda_2)$ may be
represented as a connected sum of 2-dimensional multi-fans
$\dot{\Delta}$, $\ddot{\Delta}$ depicted below:
\begin{center}
    \begin{tikzpicture}[scale=1.5]
        \draw[->,shorten >=3pt, thick] (0,0)--(1.5,-1);
        \draw[->,shorten >=3pt, thick] (1.5,-1)--(1,0);
        \draw[->,shorten >=3pt, thick] (1,0)--(1.5,1);
        \draw[->,shorten >=3pt, thick] (1.5,1)--(0,0);
        \draw[->,shorten >=3pt, thick] (3,0)--(2.5,-1);
        \draw[->,shorten >=3pt, thick] (2.5,-1)--(4,0);
        \draw[->,shorten >=3pt, thick] (4,0)--(2.5,1);
        \draw[->,shorten >=3pt, thick] (2.5,1)--(3,0);
        \fill (0,0) circle (2pt);
        \draw (-0.3,0.1) node{$\lambda(1)$};
        \fill (1,0) circle (2pt);
        \draw (0.7,0.1) node{$\lambda(2)$};
        \fill (3,0) circle (2pt);
        \draw (3.3,0.1) node{$\lambda(3)$};
        \fill (4,0) circle (2pt);
        \draw (4.3,0.1) node{$\lambda(4)$};
        \fill[red] (1.5,1) circle (2pt);
        \draw (1.5,1.2) node{$\lambda(5)$};
        \fill[red] (1.5,-1) circle (2pt);
        \draw (1.5,-1.2) node{$\lambda(6)$};
        \fill[red] (2.5,1) circle (2pt);
        \draw (2.5,1.2) node{$\lambda(5)$};
        \fill[red] (2.5,-1) circle (2pt);
        \draw (2.5,-1.2) node{$\lambda(6)$};
        \draw (2,0) node{$\hash$};
    \end{tikzpicture}
\end{center}
(The definition of connected sum is given in \cite{AyM}. There we
do not require that the set along which the connected sum is taken
is a simplex: it is only required that the values of
characteristic function on this set are linearly independent.) By
\cite[Prop.11.3]{AyM} we have
$\Hilb(\Delta_2;t)=\Hilb(\dot{\Delta};t)+\Hilb(\ddot{\Delta};t)-(1+t^4)$.
Multi-fans $\dot{\Delta},\ddot{\Delta}$ are supported on spheres,
therefore dimensions of their algebras are the $h$-vectors. These
are $(1,2,1)$ in both cases. Thus $\Hilb(\Delta_2;t)=1+4t^2+t^4$.
\end{proof}

\begin{rem}
The vertices $1,2,3,4$ of $\Gamma$ are smooth vertices. Hence the
Hilbert function of $\A^*(\Delta)$ does not depend on the values
of $\lambda$ in these vertices by Theorem \ref{thmSmoothPoints}.
Proposition \ref{propCounterEx1} implies the following
alternative: the dimensions-vector of a multi-fan on $\Gamma$ is
either $(1,2,1)$ (if the values $\lambda(5)$, $\lambda(6)$ are
collinear) or $(1,4,1)$ (if $\lambda(5)$, $\lambda(6)$ are
linearly independent).
\end{rem}

\begin{prop}\label{propComput}
There exist two complete simplicial multi-fans $\Delta_1,
\Delta_2$ which are supported on the same pseudomanifold but have
different dimensions of their multi-fan algebras.
\end{prop}

\begin{proof}
Let $L$ be the minimal triangulation of a 2-torus shown on
Fig.\ref{figTorusMin}. Using general formulas for h''-numbers, one
can show that h''-numbers of $L$ are $(1,4,4,1)$. Let $\Delta'$ be
any multi-fan supported on $L$. Since $L$ is a manifold, we have
$\Hilb(\A^*(\Delta');t)=1+4t^2+4t^4+t^6$.

\begin{figure}[h]
\begin{center}
\includegraphics[scale=0.3]{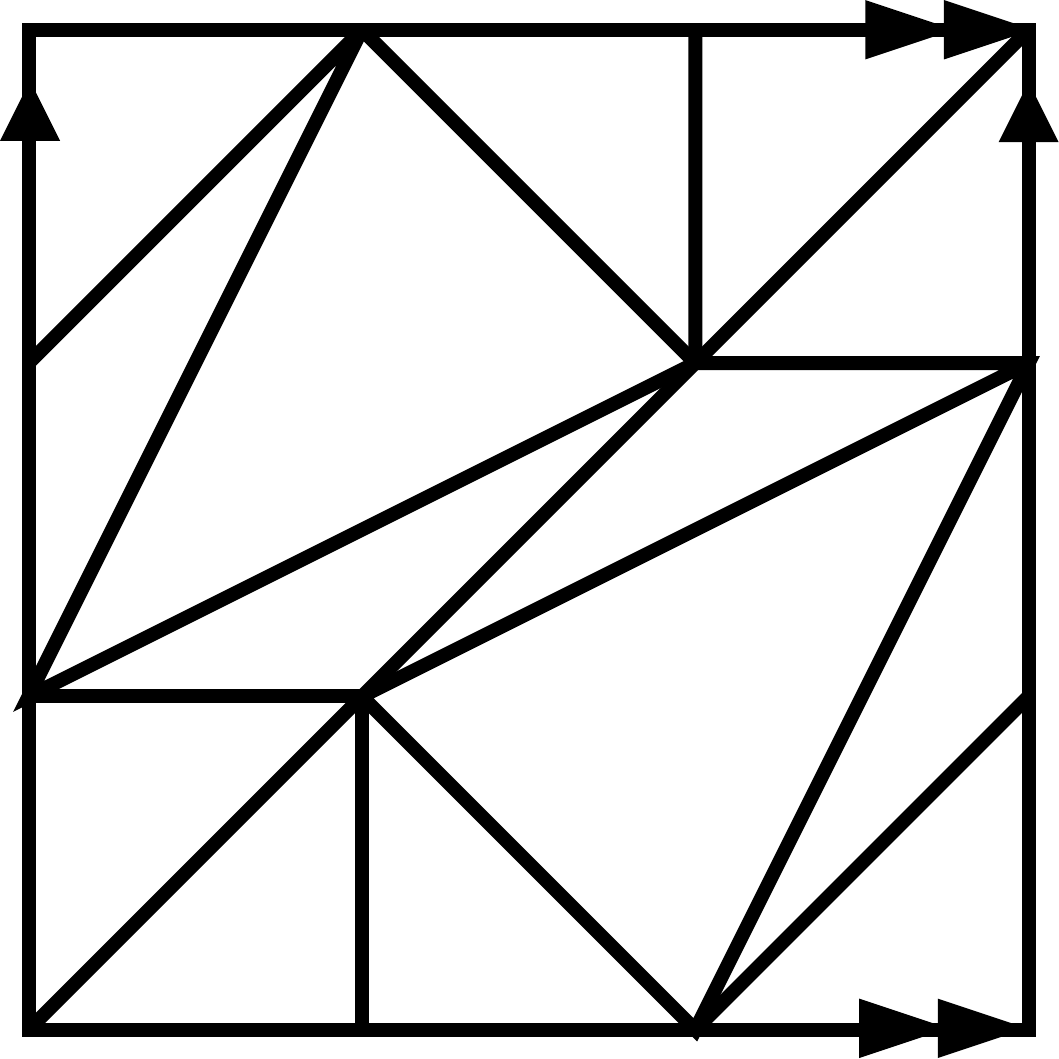}
\end{center}
\caption{Minimal triangulation of a torus}\label{figTorusMin}
\end{figure}

Consider the suspension $K=\Sigma L$. We claim that there exist
two multi-fans supported on $K$ with different d-vectors of
multi-fan algebras

(1) Let $\Delta''$ be any multi-fan supported on $S^0$. Then we
have $\Hilb(\A^*(\Delta'');t)=1+t^2$. Therefore,
\[
\Hilb(\A^*(\Delta'\ast\Delta'');t)=(1+4t^2+4t^4+t^3)(1+t^2),
\]
so the dimension vector of multi-fan $\Delta'\ast\Delta''$
supported on $K$ is (1,5,8,5,1). This multi-fan has the property
that the values of its characteristic function in suspension
points are collinear.

(2) Now we take another multi-fan supported on $K$ and set the
values of characteristic function in suspension points to be
non-collinear. In this case we do not have an easy algorithm to
compute the dimensions by hand. However, the procedure can be
easily implemented in GAP. The following programm computes the
volume polynomial and the dimension vector:

\begin{verbatim}
torus:=SCLib.Load(4);
K:=SCSuspension(genus4);
SCRelabelStandard(K);

sign:=SCOrientation(K);
n:=SCDim(K)+1;
m:=Length(SCVerticesEx(K));
MaxSimp:=SCFacets(K);

RandomVector:=function() local X; X:=[];
    for j in [1..n] do
        Add(X,Random([-15..15]));
    od;
    return X;
end;

ArrayRV:=[]; for i in [1..m+1] do
    Add(ArrayRV,RandomVector());
od;

CharFunc:=ArrayRV{[1..m]};
Polar:=ArrayRV[m+1];

#CharFunc[m]:=CharFunc[m-1];

VolPol:=0;

Clist:=[]; for i in [1..m] do
    Add(Clist,X(Rationals,Concatenation("c",ViewString(i))));
od;

r:=1;
for I in MaxSimp do
    Matr:=[];
    for i in I do
        Add(Matr,CharFunc[i]);
    od;
    alpha:=Polar*Matr^-1;
    Form:=0;
    j:=1;
    for i in I do
        Form:=Form+\alpha[j]*Clist[i];
        j:=j+1;
    od;
    D:=DeterminantMat(Matr);
    Form:=sign[r]*Form^n/Product(alpha)/D;
    r:=r+1;
    VolPol:=VolPol+Form;
od; VolPol:=VolPol/Factorial(n);

Display(VolPol);

dims:=[1]; for j in [1..n-1] do
    ListOfDerivs:=[];
    for J in UnorderedTuples([1..m],j) do
        Deriv:=VolPol;
        for a in J do
            Deriv:=Derivative(Deriv,Clist[a]);
        od;
        Add(ListOfDerivs, Deriv);
    od;
    Add(dims,Dimension(VectorSpace(Rationals,ListOfDerivs)));
od;
Add(dims,1);

Display(dims);
\end{verbatim}
The programm outputs the d-vector $(1,5,12,5,1)$ whenever the
values of $\lambda$ in suspension points (last two values of the
list) are non-collinear.
\end{proof}

In fact, there are exactly two alternatives for the dimension
vector of the suspension over a torus.

\begin{prop}\label{propTorusCase}
Let $N$ be a triangulation of a 2-torus having $m-2$ vertices, and
$K=\Sigma N$ its suspension with additional suspension points
$x,y$. For a multi-fan $\Delta=([K],\lambda)$ supported on the
pseudomanifold $K$ there are two alternatives:
\begin{enumerate}
\item d-vector equals $(1,m-4,2m-10,m-4,1)$ if the vectors
$\lambda(x), \lambda(y)$ are collinear;
\item d-vector equals $(1,m-4,2m-6,m-4,1)$ if the vectors
$\lambda(x), \lambda(y)$ are non-collinear.
\end{enumerate}
\end{prop}

\begin{proof}
At first note that all vertices of $N$, that is
$M\setminus\{x,y\}$, are smooth in the pseudomanifold $K$.
Therefore, d-vector of a multi-fan does not depend on the values
of characteristic function in these vertices by Theorem
\ref{thmSmoothPoints}.

(1) Let $M$ denote the set of vertices of $K$, $|M|=m$. Let
$\lambda(x)$,$\lambda(y)\in V\cong\Ro^4$ be collinear. We may
assume that all values $\{\lambda(i)\mid i\in M\setminus\{x,y\}\}$
lie in the 3-plane transversal to $\langle\lambda(x)\rangle$. Then
$\Delta$ is just the join of two multi-fans: one supported on $N$
and another supported on $S^0$. Formula \eqref{eqJoinMFansHilb}
implies
\[
\Hilb(\A^*(\Delta);t)=(h''_0(N)+h''_1(N)t^2+h''_2(N)t^4+h''_3(N)t^6)(1+t^2).
\]
h''-numbers of any 2-surface are easy to compute: they are
symmetric, $h_0''=1$, and $h_1''$ equals the number of vertices
minus $3$. Therefore,
\[
\Hilb(\A^*(\Delta);t)=(1+(m-5)t^2+(m-5)t^4+t^6)(1+t^2)=1+(m-4)t^2+(2m-10)t^4+(m-4)t^6+t^8.
\]
This proves the first case.

(2) Now suppose that $\lambda(x),\lambda(y)$ are non-collinear. At
first, we claim, that proposition holds for the minimal
triangulation $L$ of a torus.

Indeed, the statement holds for some particular choice of
characteristic function as was shown by direct calculation (see
the proof of Proposition \ref{propComput}). In this case we have
$m=9$ and d-vector is $(1,5,12,5,1)$. However any two
non-collinear pairs of vectors (values in suspension points) may
be translated into each other by an element $A\in\GL(V)$.
Therefore the d-vector is $(1,5,12,5,1)$ for any characteristic
function on $\Sigma L$ according to remark \ref{remLinearChange}.

Let us prove the statement for an arbitrary triangulation of a
torus. Any triangulation $N$ is connected to the minimal one by a
sequence of bistellar moves according to Pachner's theorem
\cite{Pach}. Therefore the corresponding sequence of ``suspended
bistellar moves'' joins $K=\Sigma N$ with $\Sigma L$. There are 3
types of bistellar moves in dimension 2, shown on the top of
Fig.\ref{figBistellar}. Suspended bistellar moves are shown below.

\begin{figure}[h]
\begin{center}
\includegraphics[scale=0.2]{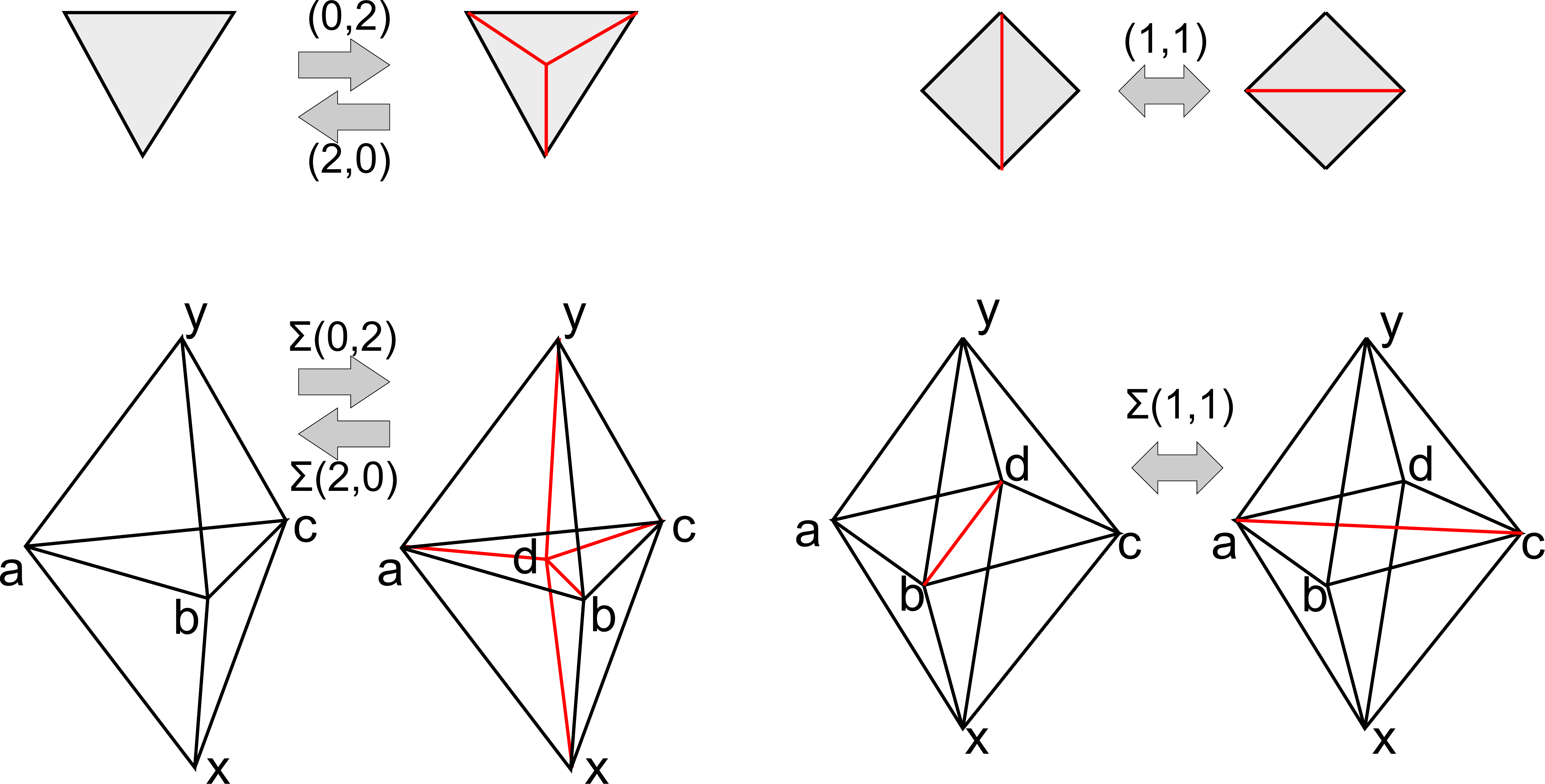}
\end{center}
\caption{Bistellar moves and their
suspensions}\label{figBistellar}
\end{figure}

Each suspended move is decomposed as a sequence of 3-dimensional
bistellar moves. In \cite[Thm.10]{AyM} we proved that bistellar
moves (otherwise called flips) performed on a multi-fan have a
very predictable effect on the dimension vector of its algebra:
the d-vector changes in exactly the same way as the h-vector does.
Let us make all the computations.

The suspended move $\Sigma(0,2)$ is the same as adding new vertex
$d$ in the tetrahedron $\{a,b,c,y\}$ (this is $(0,3)$-move)
followed by the $(1,2)$-move applied to adjacent tetrahedra
$\{a,b,c,d\}$ and $\{a,b,c,x\}$. The $(0,3)$-move adds
$(0,1,1,1,0)$ to the dimension vector and $(1,2)$-move adds
$(0,0,1,0,0)$. Therefore $\Sigma(0,2)$ increases $d_1=d_3$ by $1$
and $d_2$ by $2$. The inverse move $\Sigma(2,0)$ decreases
d-vector in the same way.

The suspended move $\Sigma(1,1)$ is equivalent to the application
of $(1,2)$-move to adjacent tetrahedra $\{a,b,d,y\}$ and
$\{b,c,d,y\}$ followed by the application of $(2,1)$-move to
tetrahedra $\{a,b,c,d\}$,$\{a,b,d,x\}$,$\{b,d,c,x\}$. First move
adds $(0,0,1,0,0)$ to the d-vector, and the second subtracts the
same value. So far, under the suspended $(1,1)$-move d-vector
remains unchanged.

In all three cases d-vector changes in the same way, as the
expression $(1,m-4,2m-6,m-4,1)$. Since d-vector is equal to this
expression for the minimal triangulation of a torus, the same
holds for any triangulation of a 2-torus.
\end{proof}

\begin{rem}
Propositions \ref{propComput}, \ref{propTorusCase} show that
suspension $K$ over a 2-torus is not rigid. A suspension-shaped
multi-fan supported on $K$ is editable if and only if the values
of its characteristic function in the suspension points are
collinear.
\end{rem}

Propositions show that multi-fans supported on a suspension of a
fixed triangulation of a torus may have two different values of
d-vector. The technique used in the proof shows that d-vectors of
multi-fans supported on suspended orientable surfaces of any genus
$g\geqslant 2$ can take no more than 2 values, depending on
whether the values of $\lambda$ in suspension points are collinear
or not. The results of calculations performed in GAP support the
following claim.

\begin{claim}\label{claimSuspendedSurface}
Let $N$ be a triangulated surface of genus $g$ with $m-2$
vertices, and $K=\Sigma N$ its suspension with additional
suspension points $x,y$. For a multi-fan $\Delta=([K],\lambda)$
supported on the pseudomanifold $K$ there are two alternatives:
\begin{enumerate}
\item d-vector equals $(1,m-4,2m-10,m-4,1)$ if the vectors
$\lambda(x), \lambda(y)$ are collinear;
\item d-vector equals $(1,m-4,2m-10+4g,m-4,1)$ if the vectors
$\lambda(x), \lambda(y)$ are non-collinear.
\end{enumerate}
\end{claim}

So far, the gap between possible values of d-vector depends on the
links of singular points. The claim was computed on examples with
$g\leqslant 10$ however we cannot explain this result
mathematically.

\section{3-dimensional pseudomanifolds}

Let $X$ be any triangulated 3-dimensional closed oriented
pseudomanifold with isolated singularities. By this we mean that
$X$ is a pure 3-dimensional simplicial complex such that links of
all its vertices are orientable surfaces and $X$ has a fundamental
cycle. We will also assume that singular points are not connected
by edges in the triangulation.

Consider the following number
\[
r(X)=\mbox{number of distinct d-vectors of multi-fans supported on
}X.
\]

\begin{prop}
$r(X)$ is a topological invariant, that is $X_1\cong X_2$ implies
$r(X_1)=r(X_2)$.
\end{prop}

\begin{proof}
Any two triangulations of a given topological pseudomanifold with
isolated singularities are connected by a sequence of bistellar
moves performed outside singularities, as was shown in
\cite[Th.4.6]{BW}. However, each bistellar move have the same
effect on all d-vectors, so the number of possible d-vectors
coincides for all triangulations.
\end{proof}

\begin{ex}
As was proved earlier, $r(X)=1$ for all manifolds. If the
pseudomanifold $X$ have only one singular point, we still have
$r(X)=1$ since the value of $\lambda$ in the singular point (the
only value that matters according to Theorem
\ref{thmSmoothPoints}) can be made arbitrary by a linear transform
of the ambient space. If $X$ is a connected sum of several
pseudomanifolds $X_i$ with $r(X_i)=1$, then $r(X)$ also equals
$1$. Indeed, in \cite{AyM} we showed that
$d_j(\Delta_1\hash\Delta_2)=d_j(\Delta_1)+d_j(\Delta_2)$ for
$j\neq 0,n$, therefore there is only one possibility for the
d-vector of connected sum, whenever this is true for the summands.
Hence there exist pseudomanifolds $X$ with any number of singular
points having $r(X)=1$.

However, $r(X)=2$ for the suspended torus (and conjecturally for
all suspended surfaces by Claim \ref{claimSuspendedSurface}). By
taking connected sums of suspended tori, we can construct
pseudomanifolds $X$ with arbitrarily large $r(X)$.
\end{ex}

\begin{ex}\label{exLinks}
Consider two 3-pseudomanifolds: $X_1$ is the suspended 2-torus and
$X_2$ is the connected sum of two copies of the space $Y$, where
$Y$ is the quotient of the solid torus by its boundary. We have
$r(X_1)=2$ by Proposition \ref{propTorusCase} and $r(X_2)=1$ by
the previous example, since the summand $Y$ have only one
singularity. The spaces $X_1$ and $X_2$ are different, but this
difference is not easy to see. The cohomology rings are
isomorphic: in both cases cohomology is torsion-free, and the
Betti numbers are $(1,0,2,1)$. The multiplication is trivial by
dimensional reasons. Also both spaces have exactly two singular
points with toric links. The difference may be seen by cutting
singular points and noticing that the first space becomes $S^3$
minus Hopf link, while the second space becomes $S^3$ minus two
unlinked circles. It is interesting that invariant $r$ can sense
such knot-theoretical distinctions.
\end{ex}

The invariant $r(X)$ somehow measures the complexity of spatial
relationships between singular points. It would be interesting to
describe this number in a more formal and computable way or at
least find out when $r(X)=1$. The next question is motivated by
example \ref{exLinks}.

\begin{problem}
Let $l\colon\bigsqcup_\alpha S_\alpha^1\hookrightarrow S^3$ be a
link (in a knot-theoretical meaning) and $X$ be a pseudomanifold
obtained by collapsing each component of $l$ to a point. Is it
true that $r(X)=1$ if and only if each two circles of the link are
unlinked?
\end{problem}

Example \ref{exLinks} shows that pairwise linking numbers affect
$r(X)$. However, the question remains: is the linking number the
only thing that matters? We wanted to test any Brunnian link.

\begin{prop}
Let $l\colon \bigsqcup_{\alpha=1,2,3} S_\alpha^1\hookrightarrow
S^3$ be the Borromean rings and $X$ be a pseudomanifold obtained
by collapsing each component of the link to a point. Then
$r(X)=1$.
\end{prop}

\begin{proof}
$X$ can be triangulated as follows. At first note that $X$ is
obtained by cutting the neighborhoods of Borromean rings from
$S^3$ and inserting the cone over each boundary component.
\begin{figure}[h]
\begin{center}
\includegraphics[scale=0.28]{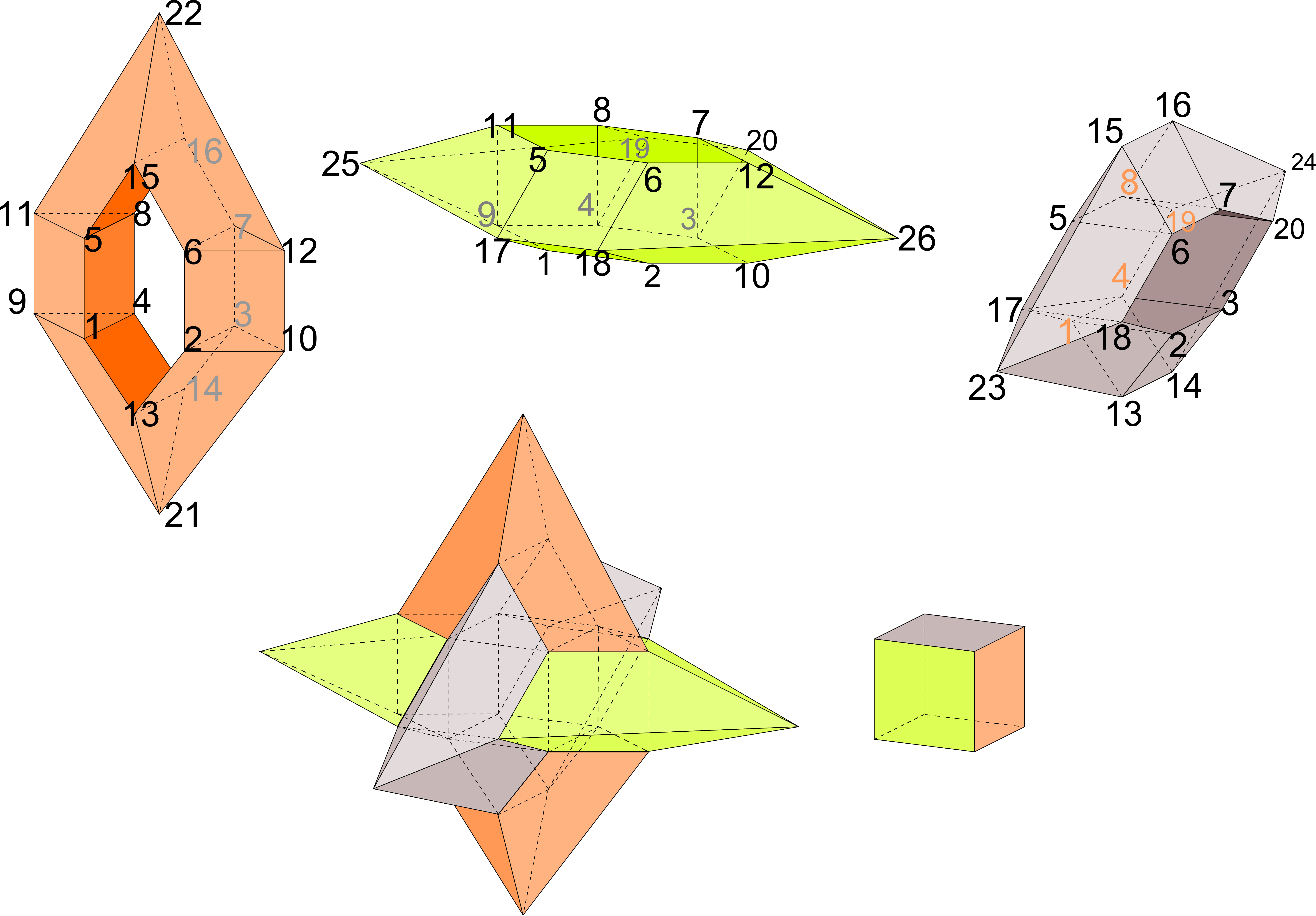}
\end{center}
\caption{Parts of construction of collapsed Borromean
rings}\label{figBorroParts}
\end{figure}

Therefore we specialize three tori in $S^3$ linked together like
Borromean rings (see left part of Fig.\ref{figLinksExamples}),
triangulate the remaining space, and put a cone over each torus.
Consider tori shown in Fig.\ref{figBorroParts} (each quadrangular
face should be further triangulated). Some parts of tori should be
identified as the labels show. The space remaining after deletion
of solid tori consists of the inner cube and the outer space (see
the lower part of Fig.\ref{figBorroParts}). The first one can be
triangulated by taking a cone with vertex in the origin, and the
second can be triangulated by taking a cone with vertex at
infinity.

\begin{figure}[h]
\begin{center}
\includegraphics[scale=0.3]{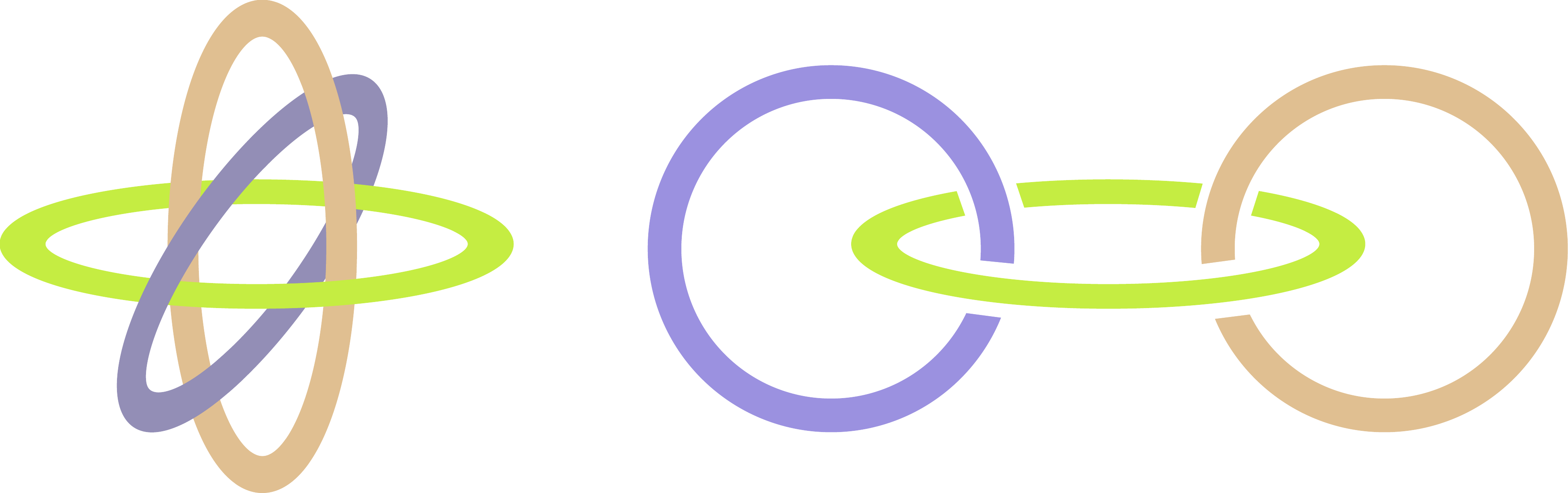}
\end{center}
\caption{Two links}\label{figLinksExamples}
\end{figure}

With figure \ref{figBorroParts} in hand the space $X$ can be
easily encoded in GAP. Now we need to specify the values of
characteristic function. Only the values in three singular points
may affect the result. In our implementation singular vertices
(which are the apices of the cones over tori) are $29,30,31$. The
symmetry group of $X$ acts transitively on the set of singular
vertices, therefore, up to linear transformation of the ambient
vector space $V$ we have the following possibilities:
\begin{enumerate}
\item $\lambda(29)$,$\lambda(30)$ and $\lambda(31)$ are linearly
independent. Without loss of generality, $\lambda(29)=(1,0,0,0)$,
$\lambda(30)=(0,1,0,0)$, $\lambda(31)=(0,0,1,0)$.
\item $\lambda(29)$,$\lambda(30)$ and $\lambda(31)$ lie in one
2-space, but any two of them are non-collinear. W.l.o.g.,
$\lambda(29)=(1,0,0,0)$, $\lambda(30)=(0,1,0,0)$,
$\lambda(31)=(1,1,0,0)$.
\item $\lambda(29)$ is collinear to $\lambda(30)$ and non-collinear
to $\lambda(31)$. W.l.o.g., $\lambda(29)=(1,0,0,0)$,
$\lambda(30)=(1,0,0,0)$, $\lambda(31)=(0,1,0,0)$.
\item $\lambda(A)$,$\lambda(B)$, and $\lambda(C)$ are collinear. W.l.o.g.,
$\lambda(29)=\lambda(30)=\lambda(31)=(1,0,0,0)$.
\end{enumerate}
All four cases are checked in GAP, and the resulting d-vector is
$(1,27,100,27,1)$ in all four cases. Therefore, for collapsed
borromean rings we have $r(X)=1$.
\end{proof}

\begin{rem}
We computed another example shown on the right part of
Fig.\ref{figLinksExamples}. After collapsing each component of
this link we obtain a pseudomanifold with 3 singular points. We
considered a special triangulation of this space, constructed
similarly to the case of Borromean rings. This triangulation has
18 vertices with singular vertices labeled $16$ (corresponds to
middle circle), $17$, and $18$. Calculations had shown that there
are 3 alternatives:
\begin{enumerate}
\item if $\lambda(16)$, $\lambda(17)$, $\lambda(18)$ are
collinear, then d-vector is $(1,14,34,14,1)$;
\item if $\lambda(16)$, $\lambda(17)$, $\lambda(18)$ span 2-dimensional space,
then d-vector is $(1,14,38,14,1)$;
\item if $\lambda(16)$, $\lambda(17)$, $\lambda(18)$ are
linearly independent, then d-vector is $(1,14,40,14,1)$.
\end{enumerate}
Surprisingly, the relation between the values of characteristic
function at singular points corresponding to unlinked circles
(vertices labeled by $17,18$) affect the answer. Indeed, when
$\lambda(17),\lambda(18),\lambda(16)$ are all in general position,
the answer differs from the case when $\lambda(17)=\lambda(18)$,
and $\lambda(16)$ is in general position. This is another strange
phenomenon which should be explained.
\end{rem}


\begin{thebibliography}{99}

\bibitem{AyM} A.\,Ayzenberg, M.\,Masuda, \textit{Volume polynomials and
duality algebras of multi-fans}, preprint
arXiv:1509.03008.



\bibitem{BW} J.W.Barrett, B.W.Westbury, \textit{Invariants of piecewise-linear
3-manifolds}, Transactions of the AMS, Vol. 348, N.10 (1996),
3997--4022.

\bibitem{simpcomp}
F.~Effenberger and J.~Spreer, \textit{{\tt simpcomp} -- a {\tt
GAP} toolkit for simplicial complexes}, Version 1.3.3, 2010,
\url{http://www.igt.uni-stuttgart.de/LstDiffgeo/simpcomp}.

\bibitem{GAP4}
  The GAP~Group, \emph{GAP -- Groups, Algorithms, and Programming,
  Version 4.8.4};
  2016,
    \url{http://www.gap-system.org}.


\bibitem{HM} A.\,Hattori, M.\,Masuda, \textit{Theory of
multi-fans}, Osaka J. Math. 40 (2003), 1--68.


\bibitem{PKh2}
A. G. Khovanskii and A. V. Pukhlikov, \textit{The Riemann--Roch
theorem for integrals and sums of quasipolynomials on virtual
polytopes}, Algebra i Analiz, 4:4 (1992), 188--216; English
transl., St. Petersburg Math. J. 4:4 (1993), 789--812.

\bibitem{Law} J.\,Lawrence, \textit{Polytope volume computation}, Math. Comp. 57 (1991),
259--271.

\bibitem{Mas} M.Masuda, \textit{Unitary toric manifolds, multi-fans and equivariant index},
Tohoku Math. J. 51 (1999), 237--265.

\bibitem{MS} D. M. Meyer, L. Smith, \textit{Poincar\'{e} Duality Algebras, Macaulay's Dual Systems, and
Steenrod Operations}, Cambridge Tracts in Mathematics, 2005.

\bibitem{NSgor} I. Novik, E. Swartz, \textit{Gorenstein rings through face rings of
manifolds}, Composit. Math. 145 (2009), 993--1000.

\bibitem{Pach} U. Pachner, \textit{P.L. homeomorphic manifolds are equivalent by elementary shellings},
European J. Combin. 12:2 (1991), 129--145.

\bibitem{Tim} V.\,A.\,Timorin, \textit{An analogue of the Hodge--Riemann
relations for simple convex polytopes}, Russian Math. Surveys
54:2, 381--426 (1999).

\end{thebibliography}
\end{document}